\documentclass[12pt]{amsart}
\DeclareFontFamily{U}{mathx}{\hyphenchar\font45}
\DeclareFontShape{U}{mathx}{m}{n}{<-> mathx10}{}
\DeclareSymbolFont{mathx}{U}{mathx}{m}{n}
\DeclareMathAccent{\widebar}{0}{mathx}{"73}
\author{Paul Pollack}
\email{pollack@uga.edu}

\author{Joseph Vandehey}
\email{vandehey@uga.edu}

\address{Department of Mathematics\\Boyd Graduate Studies Research Center\\University of Georgia\\Athens, GA 30602\\USA}

\title{Some normal numbers generated by arithmetic functions}
\usepackage{amsmath,amssymb,amsthm,url}
\usepackage{geometry}
\DeclareMathAlphabet{\curly}{U}{rsfs}{m}{n}
\renewcommand\phi\varphi

\newtheorem{thm}{Theorem}

\newtheorem*{thmA}{Theorem A}

\newtheorem{prop}[thm]{Proposition}
\newtheorem{lem}[thm]{Lemma}
\theoremstyle{definition}
\newtheorem*{df}{Definition}
\newtheorem*{rmk}{Remark}

\theoremstyle{remark}

\begin{document}
\renewcommand{\labelenumi}{(\roman{enumi})}
\def\a{\mathbf{a}}
\def\G{\curly{G}}
\def\d{\mathrm{d}}
\def\i{\mathrm{i}}
\def\e{\mathrm{e}}
\def\C{\mathbb{C}}
\def\D{\curly{D}}
\def\A{\curly{A}}
\def\Pp{\curly{P}}
\def\Prob{\mathbf{P}}
\def\E{\curly{E}}
\def\R{\mathbb{R}}
\def\M{\mathfrak{M}}
\def\N{\mathbb{N}}
\def\Q{\mathbb{Q}}
\def\Z{\mathbb{Z}}
\def\Ss{\curly{S}}
\def\1{\mathbf{1}}
\def\tD{\tilde{D}}
\def\w{\mathbf{w}}
\def\lcm{\mathop{\operatorname{lcm}}}
\newcommand{\leg}[2]{\genfrac{(}{)}{}{}{#1}{#2}}
\makeatletter
\def\moverlay{\mathpalette\mov@rlay}
\def\mov@rlay#1#2{\leavevmode\vtop{%
   \baselineskip\z@skip \lineskiplimit-\maxdimen
   \ialign{\hfil$\m@th#1##$\hfil\cr#2\crcr}}}
\newcommand{\charfusion}[3][\mathord]{
    #1{\ifx#1\mathop\vphantom{#2}\fi
        \mathpalette\mov@rlay{#2\cr#3}
      }
    \ifx#1\mathop\expandafter\displaylimits\fi}
\makeatother

\makeatletter
\let\@@pmod\pmod
\DeclareRobustCommand{\pmod}{\@ifstar\@pmods\@@pmod}
\def\@pmods#1{\mkern4mu({\operator@font mod}\mkern 6mu#1)}
\makeatother

\newcommand{\cupdot}{\charfusion[\mathbin]{\cup}{\cdot}}
\newcommand{\bigcupdot}{\charfusion[\mathop]{\bigcup}{\cdot}}
\begin{abstract} Let $g \geq 2$. A real number is said to be \emph{$g$-normal} if its base $g$ expansion contains every finite sequence of digits with the expected limiting frequency. Let $\phi$ denote Euler's totient function, let $\sigma$ be the sum-of-divisors function, and let $\lambda$ be Carmichael's lambda-function. We show that if $f$ is any function formed by composing $\phi$, $\sigma$, or $\lambda$, then the number
\[ 0. f(1) f(2) f(3) \dots \]
obtained by concatenating the base $g$ digits of successive $f$-values is $g$-normal. We also prove the same result if the inputs $1, 2, 3, \dots$ are replaced with the primes $2, 3, 5, \dots$. The proof is an adaptation of a method introduced by Copeland and Erd\H{o}s in 1946 to prove the $10$-normality of $0.235711131719\ldots$.
\end{abstract}
\maketitle

\section{Introduction}
Let $g\geq 2$. We say that a real number $\alpha$ is \emph{$g$-normal} if every preassigned sequence of digits, of length $k \geq 1$, occurs with the expected limiting frequency $g^{-k}$ in the base $g$ expansion of $\alpha$. This concept was introduced by Borel \cite{borel09} in 1909, who showed that for every $g$, almost all real numbers are $g$-normal. (Here ``almost all'' is meant in the sense of Lebesgue measure.) Regrettably, none of the more familiar mathematical constants --- such as $e$, $\pi$, or $\sqrt{2}$ --- are known to be normal to any base $g\geq 2$.

The first explicit construction of a normal number was given by Champernowne \cite{champernowne33} in 1933, while still an undergraduate. The simplest and most famous example from that paper is the (base $10$) \emph{Champernowne number}
\[ 0.12345678910111213\ldots, \]
obtained by successively concatenating the decimal digits of the positive integers. The analogous construction works for any base $g\geq 2$. This result was later extended by Copeland and Erd\H{o}s \cite{CE46}, who proved the following quite general theorem:

\begin{thmA} Let $\A$ be any set of natural numbers having the property that\begin{equation}\label{eq:CEhyp} \#\{a \in \A: a \leq x\} = x^{1-o(1)} \quad \text{as } x\to\infty. \end{equation}
Let $a_1 < a_2 < a_3 < \dots$ be the list of the elements of $\A$ in increasing order. Then for each $g\geq 2$, the number $0.a_1 a_2 a_3 \dots$ obtained by concatenating the successive base $g$ digits of the $a_i$ is $g$-normal.
\end{thmA}
\noindent
Letting $\A$ be the set of primes, this result (together with the prime number theorem) implies that the number
\[ 0.23571113171923293137\ldots \]
is $10$-normal; this answered a question left open by Champernowne.

The method of constructing normal numbers through digit concatenation has remained in vogue. In 1952, Davenport and Erd\H{o}s \cite{DE52} showed that if $f(x)$ is a nonconstant polynomial that maps positive integers to positive integers, then $0.f(1) f(2) f(3) \dots$ is normal (in any base). Forty-five years later, Nakai and Shiokawa \cite{NS97} showed that for the same class of $f$, the number $0.f(2) f(3) f(5) \dots$, with the arguments of $f$ restricted to prime values, is also normal. Analogous theorems, where $f$ is replaced by $\lfloor f\rfloor$ for certain entire functions $f$, have been given by Madritsch, Thuswaldner, and Tichy \cite{MTT08}. 

Quite recently, there has been interest in understanding the case when $f(n)$ is sensitive to the arithmetic properties of $n$. Let $P(n)$ denote the largest prime factor of $n$, with the convention that $P(1)=1$. Answering a question of Shparlinski, De Koninck and K{\'a}tai \cite{DKK11} showed that the numbers
\[ 0.P(1)P(2)P(3)P(4)\ldots \quad\text{and}\quad 0.P(2+1)P(3+1)P(5+1)P(7+1)\ldots \]
are both $g$-normal (for any $g\geq 2$). See \cite{DKK13} for further results of a similar flavor.

Vandehey \cite{vandehey13} studied the case $f(n) = \omega(n)$, where $\omega(n)$ denotes the number of distinct prime factors of $n$. Here it is important to ask the right question, for it is not reasonable to hope that $0.\omega(1)\omega(2)\omega(3) \dots$ be normal. Indeed, we expect by the Erd\H{o}s--Kac theorem that for almost all $n \leq x$, the first $\approx 50\%$ of the digits of $\omega(n)$ will coincide with the corresponding digits of $\lfloor \log\log{x}\rfloor$. However, if we let $\omega'(n)$ denote the truncated function that keeps only the last 49.9\% of the expected number of digits of $\omega(n)$, then Vandehey shows that $0.\omega'(1)\omega'(2)\omega'(3)\ldots$ is indeed a normal number.

In this note, we continue the theme of studying concatenations of digits of arithmetic functions.
Our first theorem describes a sufficient condition for normality, which we prove following the same strategy employed by Copeland and Erd\H{o}s. The precise statement requires two preliminary definitions. Throughout this paper, $\ln{x}$ denotes the natural logarithm, while $\log{x}$ denotes the function $\max\{1,\ln{x}\}$.

\begin{df} Let $f\colon \N \to \N$ be a positive-integer-valued arithmetic function. We say that $f$ is of \emph{weakly polynomial growth} if $f$ satisfies the condition \begin{equation}\label{eq:growthcond1}
 \sum_{m \leq x} \log{f(m)} \gg_{f} x\log{x}
\end{equation}
for large $x$, as well as the pointwise bound
\begin{equation}\label{eq:growthcond2}
 \log f(m) \ll_{f} \log{m}
\end{equation}
for all natural numbers $m$.
\end{df}

\begin{df} If $\E \subset \N$ is a set of positive integers, we say that $\E$ is  \emph{meager} if there is a $\delta < 1$ and a positive number $x_0$ so that whenever $x > x_0$,
\[ \#\E \cap [1,x] < x^{\delta}. \]
\end{df}

The following result gives our main tool for constructing normal numbers.

\begin{thm}\label{thm:genCE} Let $f \colon \N\to\N$ be a positive-integer-valued arithmetic function. Suppose that $f$ is of weakly polynomial growth and that the inverse image (under $f$) of any meager set is a set of asymptotic density zero. Then the number
\[ \alpha_f:= 0.f(1) f(2) f(3) \ldots  \]
is $g$-normal.
\end{thm}

\begin{rmk} If $\A \subset \N$ satisfies \eqref{eq:CEhyp}, then the hypotheses of Theorem \ref{thm:genCE} hold with $f(n)=a_n$. So Theorem A is a special case of Theorem \ref{thm:genCE}.
\end{rmk}

While Theorem \ref{thm:genCE} has its origin in the classic Copeland--Erd\H{o}s work, it has novel consequences. Let $\phi(n):=\#(\Z/n\Z)^{\times}$ be the Euler totient function, and let $\sigma(n):=\sum_{d \mid n}d$ be the usual sum-of-divisors function. Let $\lambda(n)$ denote Carmichael's lambda-function, defined as the exponent of the group $(\Z/n\Z)^{\times}$. As our main application of Theorem \ref{thm:genCE}, we produce a wide class of normal numbers arising from compositions of $\phi$, $\sigma$, and $\lambda$.

\begin{thm}\label{thm:main} Fix $g \geq 2$. Let $f\colon \N\to\N$ be any arithmetic function defined by some composition of $\phi$, $\sigma$, or $\lambda$; that is, $f = f_1 \circ f_2 \circ \dots \circ f_j$ for some $j \geq 1$, where each $f_i \in \{\phi, \sigma, \lambda\}$.
Then both of the numbers
\begin{equation}\label{eq:main1} 0.f(1) f(2) f(3) f(4)\ldots \end{equation}
	and
\begin{equation}\label{eq:main2} 0.f(2) f(3) f(5) f(7)\ldots \end{equation}
	are $g$-normal.
\end{thm}

The plan of the paper is as follows: The proof of Theorem \ref{thm:genCE} is given in \S\ref{sec:EC}. In \S\ref{sec:arith}, we present an assortment of results on $\phi$, $\sigma$, and $\lambda$ needed for the proof of Theorem \ref{thm:main}, which appears in \S\ref{sec:proofs}. We conclude the paper in \S\ref{sec:concluding} by noting other families of normal numbers that we can produce with our methods.

\subsection*{Notation} We continue to use $\omega(n)$ for the number of distinct primes dividing $n$, and we write $\Omega(n)$ for the number of primes dividing $n$ counted with multiplicity. We also write $p_n$ for the $n$th prime in the usual increasing order. If $\Ss$ is a subset of $\N$, the \emph{asymptotic density of $\Ss$} is the limit
$\lim_{x\to\infty}\frac{\#\Ss\cap[1,x]}{x}$,
if this limit exists. We use $O$ and $o$-notation, as well as the Vinogradov symbols $\ll$, $\gg$, and $\asymp$, with their usual meanings. Implied constants are absolute unless otherwise specified. We remind the reader that $\log{x} := \max\{1,\ln{x}\}$. We also use $\log_{k}{x}$ to denote the $k$th iterate of $\log{x}$. Note that with this definition, $\log_k{x}\ge 1$ whenever $x>0$.

\section{Proving normality \`a la Copeland and Erd\H{o}s}\label{sec:EC}
For the rest of this paper, we assume that the base $g\geq 2$ has been fixed once and for all. It will be convenient to have notation in place to formalize some of the concepts discussed loosely in the introduction; for this, we largely follow De Koninck and K{\'a}tai \cite{DKK11}. Let $A = \{0, 1, \dots, g-1\}$. An expression of the form $\w = a_1 a_2 \dots a_{\ell}$, where each $a_i \in A$, will be called a \emph{word of length $\ell$} on the alphabet $A$. If $n$ is a positive integer with canonical $g$-adic expansion
\begin{equation}\label{eq:gadic} n = d_0 + d_1 g + \dots + d_{t} g^{t}, \end{equation} we associate the word
\[ \overline{n} := d_0 d_1 \cdots d_t \in A^{t+1}. \]
For each positive integer $n$, we let $L(n)=t+1$ denote the length of  $\overline{n}$. Then
\begin{equation}\label{eq:lengthbound} \frac{\ln{n}}{\ln{g}} < L(n) \leq \frac{\ln{n}}{\ln{g}}+1. \end{equation} 
If $\w$ is a word on $A$, and $n$ is a positive integer, we let $\nu(n;\w)$ denote the number of occurrences of $\w$ in $\widebar{n}$. So if $\w$ has length $\ell$ and $n$ has the expansion \eqref{eq:gadic}, then $\nu(n;\w)$ is the number of times that $d_j d_{j+1} \cdots d_{g+\ell-1} = \w$ for $j=0, 1, 2, \dots, t-(\ell-1)$.

The next definition is due to Besicovitch \cite{besicovitch35}:

\begin{df} Let $\epsilon >0$ and let $k \in \N$. The positive integer $n$ is said to be \emph{$(\epsilon, k)$-normal} if for every $\w \in A^{k}$, we have
\[ (g^{-k} -\epsilon) \cdot L(n) < \nu(n;\w) <  (g^{-k} +\epsilon) \cdot L(n). \]
\end{df}

The following crucial proposition is due to Copeland and Erd\H{o}s \cite[Lemma, p. 858]{CE46}.

\begin{prop}\label{prop:EC} Let $\epsilon > 0$ and let $k \in \N$. Let $\E_{\epsilon, k}$ denote the set of positive integers that are \emph{not} $(\epsilon,k)$-normal. There is a $\delta = \delta(\epsilon, k, g) < 1$ so that
\[ \#\E_{\epsilon, k} < x^{\delta} \]
for all large $x$, say $x > x_0(\epsilon, k, g)$.
\end{prop}

We are now in a position to prove Theorem \ref{thm:genCE}. Throughout the following proof, implied constants are allowed to depend on $g$, $k$, and $f$.

\begin{proof}[Proof of Theorem \ref{thm:genCE}]
For each natural number $N$, we let $\a_{f,N}$ be the word of length $N$ obtained by truncating $\widebar{f(1)} \widebar{f(2)} \widebar{f(3)} \ldots$ after the $N$th place. Fix an arbitrary word $\w$ on the alphabet $A$, and let $k$ be the length of $\w$. Let $V$ be the number of occurrences of $\w$ in $\a_{f,N}$; we must show that
\[ V \sim \frac{N}{g^k} \quad\text{as }N\to\infty.\]
The truncation process used to create $\a_{f,N}$ cuts $\widebar{f(1)} \widebar{f(2)} \widebar{f(3)} \ldots$ either in the middle or right before the occurrence of the word $\widebar{f(n)}$ for a certain $n$. For this $n$, we have $\sum_{m=1}^{n-1} L(f(m)) \leq N \leq \sum_{m=1}^{n} L(f(m))$,
and thus
\[ 0 \leq \sum_{m=1}^{n} L(f(m)) - N \leq L(f(n)). \]
From \eqref{eq:growthcond2}, we have $L(f(n)) \ll \log{n} \ll \log{N}$. So as $N\to\infty$, \begin{equation}\label{eq:lengthasymptotic} \sum_{m=1}^{n} L(f(m)) \sim N. \end{equation}

Next, observe that $V$ is within $O(n)$ of $\sum_{m=1}^{n}\nu(f(m),\w)$. This $O(n)$ term  bounds the number of occurrences of $\textbf{w}$ that overlap multiple words $\widebar{f(m)}$, and it takes into account the error obtained by seeing only part (or none) of the word $\widebar{f(n)}$ in $\a_{f,N}$.

To proceed further, we fix $\epsilon > 0$, and we divide the integers $m \in [1,n]$ into two classes according to whether $f(m)$ is $(\epsilon,k)$-normal or not; we call $m$ \emph{good} or \emph{bad} accordingly. From \eqref{eq:lengthasymptotic},
\begin{align*} \left|\sum_{m=1}^{n} \nu(f(m),\w) - \frac{N}{g^k}\right| &\leq \left|\sum_{m=1}^{n} \nu(f(m),\w) - \frac{1}{g^k}\sum_{m=1}^{n} L(f(m))\right| + o(N) \\&\leq \sum_{m=1}^{n} \left|\nu(f(m),\w) - \frac{1}{g^k}L(f(m))\right| + o(N).	
\end{align*}
Each good value of $m$ contributes at most $\epsilon L(f(m))$ to the right-hand sum. When $m$ is bad, we use the crude estimate $\nu(f(m),\w) \leq L(f(m)) \ll \log{m}$ to see that each bad $m$ contributes $O(\log{n})$. Putting everything together, we deduce that
\begin{equation}\label{eq:Vdiff0}\left|V - \frac{N}{g^k}\right| \leq \epsilon \sum_{m=1}^{n}L(f(m)) + O(\log{n} \cdot \#\{\text{bad m}\}) + O(n) + o(N). \end{equation}
From the growth bounds \eqref{eq:growthcond1} and \eqref{eq:growthcond2}, we have $\sum_{m=1}^{n}L(f(m)) \asymp n\log{n}$, and so \eqref{eq:lengthasymptotic} shows that
\[ N \asymp n\log{n} \]
for large $N$. So dividing \eqref{eq:Vdiff0} by $N$, we obtain for $N$ going to infinity that
\begin{align*} \left|\frac{V}{N} - \frac{1}{g^k}\right| &\leq \epsilon \left(\frac{1}{N} \sum_{m=1}^{n}L(f(n))\right) + O\left(\frac{1}{N}\log{n} \cdot \#\{\text{bad m}\}\right) + O\left(\frac{n}{N}\right) + o(1) \\
&\leq \epsilon \cdot (1+o(1)) + O\left(\frac{1}{n}\cdot \#\{\text{bad m}\}\right) + o(1).
\end{align*}
Now $m$ is bad precisely when $m \in f^{-1}(\E_{\epsilon, k})$. By Proposition \ref{prop:EC}, the set $\E_{\epsilon,k}$ is meager, and so all bad $m$ are restricted to a set of density zero. Consequently, the number of bad $m$ in $[1,n]$ is $o(n)$ as $N\to\infty$. Hence,
\[ \limsup_{N\to\infty} \left|\frac{V}{N} - \frac{1}{g^k}\right| \leq \epsilon. \]
Since $\epsilon >0$ was arbitrary, it follows that $V \sim N/g^k$ as $N\to\infty$.
\end{proof}

We emphasize that the growth conditions \eqref{eq:growthcond1} and \eqref{eq:growthcond2} are not necessary for this proof to work, and many other growth conditions would work in their place.

To prove Theorem \ref{thm:main}, we will show that if $f$ is an arbitrary composition of $\phi$, $\sigma$, and $\lambda$,  then the hypotheses of Theorem \ref{thm:main} hold for both $f(n)$ and $f(p_n)$.

\section{Arithmetic preparation}\label{sec:arith}

Here we collect some lemmas needed for the eventual proof of Theorem \ref{thm:main}. We begin with the well-known determinations of the minimal order of the Euler function \cite[Theorems 328, p. 352]{HW08} and the maximal order of the sum-of-divisors function \cite[Theorems 323, p. 350]{HW08}.

\begin{lem}\label{lem:minimalorder} Let $\gamma = 0.5772156649\dots$ denote the Euler--Mascheroni constant. Then
\[ \liminf_{m\to\infty} \frac{\phi(m)}{m/\log_2{m}} = e^{-\gamma}, \]
and
\[ \limsup_{m\to\infty} \frac{\sigma(m)}{m \log_2{m}} = e^{\gamma}. \]
\end{lem}

The function $\lambda(n)$ is more erratic than $\phi$ or $\sigma$ and occasionally takes values as small as $n^{o(1)}$. The following result serves as a substitute for Lemma \ref{lem:minimalorder} in this case.

\begin{lem}\label{lem:FPS} The number of $n \leq x$ where $\lambda(n) < n^{1/2}$ is at most
	\[ x/\exp((\log{x})^{1/3}) \]
	for all large $x$.
\end{lem}
\begin{proof} A theorem of Friedlander, Pomerance, and Shparlinski \cite[Theorem 5]{FPS01} asserts that for all large $x$ and for $\Delta \geq (\log\log{x})^3$, the number of $n \leq x$ with
\[ \lambda(n) \leq n \exp(-\Delta) \]
is at most
\begin{equation}\label{eq:FPS} x/\exp(0.69 (\Delta \log \Delta)^{1/3}). \end{equation}
If $n > x^{2/3}$ but $\lambda(n) < n^{1/2}$, then $\lambda(n) < n \exp(-\Delta)$ for $\Delta := \frac{1}{3}\log{x}$. Using this in \eqref{eq:FPS}, we get that the number of $n \leq x$ with $\lambda(n) < n^{1/2}$ is eventually bounded by
\[ x^{2/3}+ x/\exp(0.69 (\Delta \log \Delta)^{1/3}) < x/\exp((\log{x})^{1/3}). \qedhere \]
\end{proof}

\begin{lem}\label{lem:phidivupper} Let $a(n)$ be any of the functions $\phi(n)$, $\sigma(n)$, or $\lambda(n)$. Let $d$ be a positive integer, and let $\ell:= \Omega(d)$. For each $x \geq 1$, the number of $n \leq x$ where $d \mid a(n)$ is at most
\begin{equation}\label{eq:divbound} \frac{x}{d} (8\ell \log^2{x})^{\ell}.  \end{equation}
\end{lem}

\begin{proof} Since $\lambda(n) \mid \phi(n)$ for all $n$, we can (and do) assume that $a(n)$ is one of $\phi(n)$ or $\sigma(n)$. Suppose now that $d \mid a(n)$, where the prime factorization of $n$ is $p_1^{e_1} \cdots p_r^{e_r}$. Since $d \mid \prod_{i=1}^{r} a(p_i^{e_i})$, we can write $d = d_1 d_2 \cdots d_{r}$ where each $d_i \mid a(p_i^{e_i})$, for $i=1, 2, \dots, r$. By reordering if necessary, we can assume that $d_i > 1$ precisely for $i=1, 2, \dots, k$, say.  Given the factorization $d=d_1 d_2 \cdots d_k$ and the prime powers $p_1^{e_1}, \dots, p_k^{e_k}$, the number of corresponding $n=p_1^{e_1}\cdots p_r^{e_r}\le x$ is at most
\[ \frac{x}{p_1^{e_1} \cdots p_k^{e_k}}. \]
Keeping the $d_i$ fixed, we sum over the possible choices for the prime powers $p_i^{e_i}$. Since $p_i^{e_i} \leq n \leq x$, each $a(p_i^{e_i}) \leq \sigma(p_i^{e_i}) = 1 + p_i + \dots + p_i^{e_i} < p_i^{e_i} (1+1/p + 1/p^2 + \dots) \leq 2p_i^{e_i}\le 2x$. So we get an upper bound of
\begin{equation}\label{eq:UB} x \prod_{i=1}^{k}\sum_{\substack{p_i^{e_i}:~a(p_i^{e_i}) \leq 2x \\ d_i \mid a(p_i^{e_i})}}\frac{1}{p_i^{e_i}}. \end{equation}
Turning to the inner sum, we observe that
\begin{equation}\label{eq:innersum} \sum_{\substack{p_i^{e_i}:~a(p_i^{e_i}) \leq 2x \\ d_i \mid a(p_i^{e_i})}}\frac{1}{p_i^{e_i}} =  \sum_{\substack{m \leq 2x \\ d_i \mid m}} \sum_{a(p_i^{e_i})=m} \frac{1}{p_i^{e_i}} \leq 2\sum_{\substack{m \leq 2x \\ d_i\mid m}}\frac{1}{m} \sum_{a(p_i^{e_i})=m} 1.\end{equation}
Now for each fixed $e$, the expression $a(p^e)$ is a strictly increasing function of the prime variable $p$, and so there is at most one value of $p$ with $a(p^e)=m$. Moreover, if $e > 2+\frac{\ln{x}}{\ln{2}}$, then $p^e > 4x$ for each prime $p$, and so $a(p^e) > 2x \geq m$. Hence, the rightmost inner sum in \eqref{eq:innersum} is at most $2+\frac{\ln{x}}{\ln{2}} < 4\log{x}$, say. Also,
\[ \sum_{\substack{m \leq 2x \\ d_i\mid m}}\frac{1}{m} \leq \frac{1}{d_i}\sum_{m' \leq x}\frac{1}{m'} \leq \frac{1}{d_i}(1+\ln{x}) \leq \frac{2\log{x}}{d_i}. \]
Collecting these estimates, we see that the first expression in \eqref{eq:innersum} is at most $\frac{8\log^2{x}}{d_i}$. Putting this back into \eqref{eq:UB}, our upper bound does not exceed
\begin{equation}\label{eq:UB2} x \prod_{i=1}^{k}\frac{8\log^2{x}}{d_i} = \frac{x}{d} (8 \log^2{x})^k \leq \frac{x}{d} (8\log^2{x})^{\ell}; \end{equation}
here the final inequality uses that $k \leq \sum_{i=1}^{k} \Omega(d_i) = \Omega(d) = \ell$. Finally, we sum over the number of possibilities for the (unordered) factorization $d_1 \cdots d_k$ of $d$; this is crudely bounded above by $\ell^{\ell}$. Inserting this factor into \eqref{eq:UB2} gives the bound \eqref{eq:divbound}.
\end{proof}
\begin{rmk} This argument is based on the proof of \cite[Lemma 3.6]{PT13}. Lemma \ref{lem:phidivupper} might also be compared with \cite[Lemma 2]{BKW99} and \cite[Lemma 2.1]{LP11}. \end{rmk}

\begin{lem}\label{lem:toomanyprimes} Let $a(n)$ be any of the functions $\phi(n)$, $\sigma(n)$, or $\lambda(n)$. Take any integer $K \geq 1$. For $x\geq 1$, the number of positive integers $n \leq x$ with
\[ \Omega(a(n)) > K^2 \]
is
\begin{equation}\label{eq:ourupper} \ll \frac{K}{2^K} x (\log{x})^3. \end{equation}
\end{lem}

\begin{proof} Again, we may assume that $a(n)$ is either $\phi(n)$ or $\sigma(n)$. We begin by recalling Lemma 13 of \cite{LP07} (due to Hall and Tenenbaum), asserting that
\begin{equation}\label{eq:T} \sum_{\substack{m \leq t \\ \Omega(m)\geq K}} 1 \ll \frac{K}{2^K} t\log{t}, \end{equation}
uniformly for real $t\geq 1$ and positive integers $K$. Since $\Omega(a(n)) = \sum_{p^e\parallel n} \Omega(a(p^e))$, then if $\Omega(a(n)) > K^2$, we have either
\begin{enumerate}
\item $\Omega(n) \geq \omega(n) > K$, \emph{or}
\item  there is a prime power $p^e \parallel n$ with $\Omega(a(p^e)) > K$.
\end{enumerate}  From \eqref{eq:T} with $t=x$, the the number of $n\leq x$ where (i) holds is $O(\frac{K}{2^K} x\log{x})$, which is acceptable. Now the number of $n \leq x$ where (ii) holds is at most
\begin{align} \notag x\sum_{\substack{p^e \leq x \\ \Omega(a(p^e)) > K}} \frac{1}{p^e} &\leq 2x \sum_{\substack{p^e \leq x\\ \Omega(a(p^e)) >K}}\frac{1}{a(p^e)} \\ &\leq 2x \sum_{\substack{m \leq 2x \\ \Omega(m) > K}}\frac{1}{m} \sum_{a(p^e)=m} 1 \notag\\ &\leq 8x\log{x} \sum_{\substack{m \leq 2x \\ \Omega(m) > K}}\frac{1}{m}; \label{eq:primepowercase}\end{align}
in moving from the second line to the third, we have used that the number of prime powers $p^e$ with $a(p^e)=m$ is bounded by $4\log{x}$, exactly as in the proof of Lemma \ref{lem:phidivupper}. To estimate the remaining sum, we use \eqref{eq:T} along with partial summation:
\begin{align*} \sum_{\substack{m \leq 2x \\ \Omega(m) > K}}\frac{1}{m} &= \int_{1}^{2x} \frac{1}{t}\,\d{\left(\sum_{\substack{m\leq t\\ \Omega(m)> K}} 1\right)} \\&\ll \frac{K}{2^K} \left(\log{(2x)} + \int_{1}^{2x} \frac{\log{t}}{t}\, \d{t}\right) \ll \frac{K}{2^K}(\log{x})^2.\end{align*}
Inserting this back into \eqref{eq:primepowercase} yields the claimed upper bound \eqref{eq:ourupper}.
\end{proof}

\section{Proof of Theorem \ref{thm:main}}\label{sec:proofs}

We need one more definition.

\begin{df} If $\E \subset \N$ is a set of positive integers, we say that $\E$ is \emph{thin} if there is a $\theta > 0$ and a positive number $x_0$ so that whenever $x > x_0$,
	\[ \#\E \cap [1,x] < x/\exp((\log{x})^{\theta}). \]
\end{df}

Note that every meager set is thin and that every thin set is of density zero. Our key lemma is the following:

\begin{lem}\label{lem:smallpreimage} Let $\E$ be a thin set of positive integers. Let $a(n)$ be any of the functions $\phi(n)$, $\sigma(n)$, or $\lambda(n)$. Then $a^{-1}(\E)$ is also a thin set.
\end{lem}

\begin{proof} Fix $\theta > 0$ so that the number of elements of $\E$ not exceeding $t$ is bounded by $t/\exp((\log{t})^{\theta})$ for all sufficiently large values of $t$. For large real numbers $x$, let us estimate the number of $n \leq x$ with $a(n) \in \E$. We partition $\E$ into sets
\begin{enumerate}
\item $\E_1 = \{m \in \E: m \leq x^{1/3}\}$,
\item $\E_2 = \{m \in \E \setminus \E_1: \Omega(m) > (\log{x})^{\theta/3} \}$,
\item $\E_3 = \E \setminus (\E_1 \cup \E_2)$.
\end{enumerate}
If $a=\lambda$ and $a(n) \leq x^{1/3}$, then either $n \leq x^{2/3}$ or $\lambda(n) < n^{1/2}$; so by Lemma \ref{lem:FPS}, $a(n) \in \E_1$ for at most
\begin{equation}\label{eq:ainverse1} x/\exp((\log{x})^{1/4})\end{equation} values of $n\leq x$, once $x$ is large. This bound also holds for $a=\phi$ or $\sigma$; indeed, in these cases, one has the much stronger result that $a(n) > x^{1/3}$ whenever $n > x^{1/3}\log{x}$. Now suppose that $a(n) \in \E_2$. Taking $K = \lfloor (\log{x})^{\theta/6}\rfloor$ in Lemma \ref{lem:toomanyprimes}, we find after a brief computation that the number of such $n\leq x$ is at most
\begin{equation}\label{eq:ainverse2} x/\exp((\log{x})^{\theta/7}), \end{equation}
for large enough values of $x$. Finally, suppose that $a(n) = m$ for an $m \in \E_3$. Let $\ell:=\Omega(m)$. Since $m \mid a(n)$, Lemma \ref{lem:phidivupper} shows that the number of these $n$ is at most
\[ \frac{x}{m} (8\ell \log^2{x})^{\ell} \leq \frac{x}{m} \exp((\log{x})^{\theta/2}) \]
for large $x$. Here we have used that $\ell \leq (\log{x})^{\theta/3}$, since $m\in \E_3$. Summing over $m \in \E_3$ gives an upper bound on the total number of these $n$ that does not exceed
\begin{align*} x \exp((\log{x})^{\theta/2}) \cdot \sum_{\substack{m > x^{1/3} \\ m \in \E}} \frac{1}{m} &\leq x \exp((\log{x})^{\theta/2}) \cdot \int_{x^{1/3}}^{\infty} \frac{1}{t^2} \cdot \#\{m \in \E: m \leq t\} \,\d{t} \\
	& \leq x \exp((\log{x})^{\theta/2}) \cdot \int_{x^{1/3}}^{\infty} \frac{1}{t \exp((\log{t})^{\theta})} \,\d{t}.\end{align*}
The final integral is eventually smaller than $\exp(-(\log{x})^{9\theta/10})$, say, and thus the total number of $n \leq x$ with $a(n)\in \E_3$ is eventually smaller than
\begin{equation}\label{eq:ainverse3} x/\exp((\log{x})^{4\theta/5}), \end{equation}
for instance. Adding \eqref{eq:ainverse1}, \eqref{eq:ainverse2}, and \eqref{eq:ainverse3}, we see that the size of $a^{-1}(\E) \cap [1,x]$ is eventually smaller than $x/\exp((\log{x})^{\eta})$ for any fixed
\[ \eta < \min\left\{\frac{1}{4}, \frac{\theta}{7}, \frac{4\theta}{5}\right\}. \]
Hence, $a^{-1}(\E)$ is a thin set.
\end{proof}

\begin{lem}\label{lem:usuallylarge} Suppose that $f = f_1 \circ f_2 \circ \dots \circ f_j$, where each $f_i \in \{\phi, \sigma, \lambda\}$. Then the set of $n$ where
\[ f(n) < n^{\frac{1}{2^j}} \]
is a thin set.	
\end{lem}
\begin{proof} Let $\E$ be the set of $n$ where either $\phi(n) < n^{1/2}$ or $\lambda(n) < n^{1/2}$. There are only finitely many $n$ with $\phi(n) < n^{1/2}$, which together with Lemma \ref{lem:FPS} implies that $\E$ is a thin set.

Suppose now that $f(n) < n^{1/2^j}$. For each $0 \leq i < j$, put $h_i := f_{j-i} \circ f_{j-i+1} \circ \dots \circ f_j$. There is at least one index $i$ with $0 \leq i < j$ having $h_i(n) < n^{1/2^{i+1}}$, namely $i=j-1$. Now select the smallest such $i$. If $i=0$, then $f_j(n) < n^{1/2}$; so since $\sigma(n) \ge n$, we have $f_j \in \{\phi, \lambda\}$ and $n$ belongs to the thin set $\E$. If $1 \leq i < j$, set  $m:=h_{i-1}(n)$. Then $m \geq n^{1/2^i}$, while \[ f_{j-i}(m) = h_i(n)  < n^{1/2^{i+1}} \leq m^{1/2}.\] Thus, again since $\sigma(m)\ge m$, we have $f_{j-i} \in \{\phi,\lambda\}$ and $m \in \E$. Hence, $n \in h_{i-1}^{-1}(\E)$. Applying Lemma \ref{lem:smallpreimage} repeatedly, we see that $h_{i-1}^{-1}(\E)$ is a thin set.  We have shown that any solution to the inequality $f(n) < n^{1/2^j}$ belongs to
\[ \E \cup \bigcup_i h_{i-1}^{-1}(\E) , \]
where $i$ runs over all values $1\le i < j$ for which $f_{j-i}\in\{\phi, \lambda\}$. This is a finite union of thin sets and hence thin itself.
\end{proof}

\begin{proof}[Proof of Theorem \ref{thm:main}] Suppose that $f = f_1 \circ f_2 \circ \dots \circ f_j$, where each $f_i \in \{\phi, \sigma, \lambda\}$. Since $\lambda$ and $\phi$ map their inputs to smaller values, the maximal order of the sum-of-divisors function (see Lemma \ref{lem:minimalorder}) yields $f(n) \leq 2^j n (\log_2{n})^j$ for all large values of $n$. So for large enough $n$,
\[ \log f(n) \leq 2\log{n}. \]
Thus, the growth condition \eqref{eq:growthcond2} holds for $f(n)$. We turn now to the other growth condition \eqref{eq:growthcond1}. From Lemma \ref{lem:usuallylarge}, the set of $n$ with $f(n) < n^{1/2^j}$ is a thin set. Consequently,
\begin{align} \notag\sum_{n \leq x}\log{f(n)} &\geq \sum_{n \leq x} \ln f(n) \geq \frac{1}{2^j} \sum_{\substack{n \leq x\\ f(n) \geq n^{1/2^j}}}\ln{n} \\&\geq \frac{1}{2^j} \Bigg(x\log{x}+O(x) + O\bigg(\log{x} \sum_{\substack{n \leq x \\ f(n) < n^{1/2^j}}}1\bigg)\Bigg) = \left(\frac{1}{2^j}+o(1)\right)x \log{x},\label{eq:lowerboundcomputation}  \end{align}
as $x\to\infty$, verifying \eqref{eq:growthcond1} for $f(n)$. Repeated application of Lemma \ref{lem:smallpreimage} shows that the preimage of a thin set under $f$ is thin. In particular, the preimage of a meager set is thin, and so of density zero. Hence, Theorem \ref{thm:genCE} applies to $f(n)$, and the number \eqref{eq:main1} is normal.

Now consider the function of $n$ given by $f(p_n)$, where $p_n$ is the $n$th prime. We view this as the composition $f\circ \iota$, where $\iota(n) = p_n$. By the prime number theorem, $p_n \sim n\log{n}$ as $n\to\infty$, and so from the preceding paragraph,
\[ \log f(p_n) \leq 2\log{p_n} < 3\log{n} \]
for all large $n$, which shows that $f\circ \iota$  satisfies condition \eqref{eq:growthcond2}. The relation $p_n \sim n\log{n}$ also implies that the preimage of a thin set under $\iota$ is  a thin set. Since $f(n) \geq n^{1/2^j}$ except on a thin set, we see that
\[ f(p_n) > p_n^{1/2^j} > n^{1/2^j} \]
for all $n$ outside of a thin set. Mimicking the lower bound computation \eqref{eq:lowerboundcomputation}, we deduce that $f \circ \iota$ also satisfies \eqref{eq:growthcond1}. It remains to show that meager sets have density zero preimages. If $\E$ is a meager set, we have already seen that $f^{-1}(\E)$ is thin; hence, $(f\circ \iota)^{-1}(\E) = \iota^{-1}(f^{-1}(\E))$ is also thin, and in particular of density zero. This completes the proof of applicability of Theorem \ref{thm:genCE} and also the proof of normality of \eqref{eq:main2}.
\end{proof}

\begin{rmk} One can modify this argument to prove the following generalization of Theorem \ref{thm:main}, very much in the spirit of Theorem A: \emph{Let $f$ be any composition of $\phi$, $\sigma$, and $\lambda$. Suppose that $\Ss$ is a set of natural numbers with the property that
\[ \#\Ss \cap [1,x] > \frac{x}{(\log{x})^B} \]
for a certain  constant $B$ and all large enough $x$. List the elements of $\Ss$ as $s_1 < s_2 < s_3 < \dots$. Then the real number $0.\widebar{f(s_1)} \widebar{f(s_2)} \widebar{f(s_3)} \ldots$ is $g$-normal.}
\end{rmk}

\section{Concluding thoughts}\label{sec:concluding}
Theorem \ref{thm:main} is only one of several possible applications of Theorem \ref{thm:genCE}. Here we report on some other families of $g$-normal numbers that can be produced either directly from Theorem \ref{thm:genCE} or by following its proof.

\subsection{Multiplicative functions whose values divide their arguments} Suppose that $f\colon\N\to\N$ is a multiplicative function having the property that $f(n)\mid n$ for each natural number $n$. For example, $f(n)$ might be the radical of $n$ (that is, $\prod_{p \mid n}p$) or the largest divisor of $n$ expressible as a sum of two squares. With  \begin{equation}\label{eq:lotsofgood}\G := \{p: f(p)=p\}, \quad\text{suppose that}\quad \sum_{\substack{p \leq t \\ p \in \G}} \frac{\log{p}}{p} \to \infty \quad \text{as $t\to\infty$}.  \end{equation}
We claim that $0.\widebar{f(1)}\widebar{f(2)} \widebar{f(3)}\dots$ is $g$-normal.

Note that \eqref{eq:lotsofgood} holds even for reasonably sparse sets of primes $\G$; for instance, it is sufficient that $\#\G \cap [1,t] \gg t/(\log{t})^2$ for large $t$.

Let $N$ and $n$ have the same meaning as in the proof of Theorem \ref{thm:genCE}. Carefully reading that argument, we see it suffices to prove the following estimate for any meager set $\E$: As $N$ (and hence also $n$) tends to infinity,
\begin{equation}\label{eq:forany} \sum_{\substack{1\leq m \leq n\\ f(m) \in \E}} L(f(m)) = o\left(\sum_{m=1}^{n} L(f(m))\right). \end{equation}
In fact, we will show that the left-hand side is $O_{\E}(n)$ while the right-hand side exceeds any constant multiple of $n$ for large enough $n$.

To handle the left-hand side of \eqref{eq:forany}, we use that $f(m) \mid m$ to see that
\begin{align*} \sum_{\substack{1\leq m \leq n\\ f(m) \in \E}} L(f(m)) &\ll \sum_{\substack{1\leq m \leq n\\ f(m) \in \E}} \log f(m) \\&\leq \sum_{\ell \in \E} (\log \ell) \sum_{\substack{n \le m \\ \ell \mid n}} 1\leq n \sum_{\ell \in \E} \frac{\log{\ell}}{\ell}. \end{align*}
A straightforward exercise in partial summation shows that the sum of $\frac{\log{\ell}}{\ell}$ converges for $\ell$ in any meager set. So the left-hand side of \eqref{eq:forany} is $O_{\E}(n)$.

To estimate the right-hand side of \eqref{eq:forany}, we note that
\begin{align}\notag \sum_{m \leq n}L(f(m)) &\gg_{g} \sum_{m \leq n}\ln{f(m)} \\&\geq \sum_{\substack{m \leq n \\ m\text{ squarefree}}} \sum_{p \mid m} \ln f(p) \notag\\&= \sum_{p \le n} \ln f(p) \sum_{\substack{m \leq n \\ m \text{ squarefree} \\ p \mid m}} 1.\label{eq:fdivverify}
\end{align}
Consider the contribution to \eqref{eq:fdivverify} from primes $p \leq \sqrt{n}$. For $m \leq n$ to be squarefree and divisible by $p$, we need that $m=pm'$ for some squarefree $m' \leq n/p$ not divisible by $p$. The total number of squarefree $m' \leq n/p$ is $\sim \frac{6}{\pi^2}n/p$ as $n\to\infty$, while the total number of $m'\leq n/p$ that are multiples of $p$ is at most $n/p^2 \leq \frac{1}{2} n/p$. Since $\frac{6}{\pi^2} - \frac{1}{2} > \frac{1}{10}$ (say), the contribution to \eqref{eq:fdivverify} from primes $p \leq \sqrt{n}$ is eventually \[ \gg n \sum_{p \leq \sqrt{n}} \frac{\ln{f(p)}}{p} \gg n\sum_{\substack{p \leq \sqrt{n} \\ p \in \G}} \frac{\ln{p}}{p}; \]
but \eqref{eq:lotsofgood} shows that the final sum on $p$ tends to infinity with $n$.

\subsubsection{A non-normality result}

Suppose again that $f\colon\mathbb{N}\to \mathbb{N}$ is a multiplicative function having the property that $f(n) \mid n$ for each natural number $n$. However, in this case, suppose that $\G := \{p: f(p)=p\}$ is a finite set and moreover that if $p\not\in \G$ then $f(p^k)=1$. We claim that $0.\widebar{f(1)}\widebar{f(2)} \widebar{f(3)}\dots$ is \emph{not} $g$-normal in this case.

Let us assume by way of contradiction that $0.\widebar{f(1)}\widebar{f(2)} \widebar{f(3)}\dots$ is $g$-normal.

The length of the word $\widebar{f(1)}\widebar{f(2)} \dots \widebar{f(n)}$ now behaves rather differently, as $n\to\infty$. Recalling \eqref{eq:lengthbound}, we see that
\begin{align}\notag \sum_{1 \leq m \leq n} L(f(m)) &\ll \sum_{1 \leq m \leq n}\left(1 + \ln{f(m)}\right) \\
&= n + \sum_{1 \leq m \leq n} \sum_{\substack{p^k\parallel m \\ p\in \G }}\ln(p^k) \notag\\
&\leq n + n \sum_{\substack{p^k\\ p \in \G}} \frac{\ln{p^k}}{p^k} \ll n. \label{eq:nonnormalitygrowth}\end{align}

Let $M=\prod_{p\in \G} p$ and let $k$ be a large positive integer. The strings
\[
 \overline{f(1+j M^k)}~\overline{f(2+j M^k)}~\overline{f(3+j M^k)} \dots \overline{f(2^k-1 + j M^k)}, \qquad j \in \mathbb{N},
\]
are the same for all $j$, because $f(i) = f(i+j M^k)$ for $1 \le i \le 2^k-1$, unless there is some prime $p\in \G$ that divides $i$ and $i+jM^k$ to different powers, which is clearly impossible for any $j$. Therefore by \eqref{eq:nonnormalitygrowth}, the number of occurrences of $\widebar{f(1)} \widebar{f(2)} \widebar{f(3)} \dots \overline{f(2^k-1 )}$ in the first $n$ digits of  $0.\widebar{f(1)}\widebar{f(2)} \widebar{f(3)}\dots$ is $\gg_f n/M^k$ for sufficiently large $n$.

On the other hand, the string $ \widebar{f(1)} \widebar{f(2)} \widebar{f(3)} \dots \overline{f(2^k-1 )}$ contains at least $2^k-1$ digits, so by the normality assumption, it should appear in the first $n$ digits at most
\[
 \frac{n}{g^{2^k-1}} (1+o(1))
\]
times, as $n\to\infty$. Taking $k$ large enough gives the desired contradiction.

We note that this leaves open the question of normality when $\G$ is infinite but $\sum_{p\in \G} \log p/ p < \infty$.

\subsection{Orders} The function in this example is a close relative of $\phi$ and $\lambda$. For odd numbers $n$, let $\ell(n)$ denote the multiplicative order of $2$ mod $n$. We claim that $f(n)=\ell(2n-1)$ satisfies the hypotheses of Theorem \ref{thm:genCE}, and so $0. \widebar{\ell(1)} \widebar{\ell(3)} \widebar{\ell(5)} \widebar{\ell(7)}\dots$
is $g$-normal.

Let us quickly see why. It is known that for all but $o(x)$ integers $n \leq x$ (as $x\to\infty$), we have $\ell(2n-1) > x^{1/3}$;  see, for example, \cite[Theorem 17]{KR01}. This implies the growth condition \eqref{eq:growthcond1}, while \eqref{eq:growthcond2} follows from the trivial bound $\ell(2n-1)< 2n$. Next, we show that thin sets (and so also meager sets) have density zero preimages. Let $\E$ be a thin set, and choose $\theta > 0$ so that $\#\E \cap [1,t] \leq t/\exp((\log{t})^{\theta})$ for large enough $t$. Partition $\E$ into sets $\E_1$, $\E_2$, and $\E_3$ defined exactly as in the proof of Lemma \ref{lem:smallpreimage}. As remarked above, only $o(x)$ integers $n\leq x$ have $f(n) \in \E_1$. Suppose now that $f(n) \in \E_2$. Since $\ell(n)$ divides $\phi(2n-1)$, we see that $\Omega(\phi(2n-1)) > (\log{x})^{\theta/3}$, and reasoning as in the proof of Lemma \ref{lem:smallpreimage}, the number of these $n$ is eventually bounded by $x/\exp((\log{x})^{\theta/7})$. Finally, suppose $m \in \E_3$ and that $f(n)=m$. Using that $m \mid \phi(2n-1)$ and proceeding as in the proof of Lemma \ref{lem:smallpreimage}, we get that $2n-1$ is restricted to a set of size at most $x/\exp((\log{x})^{4\theta/5})$ for large $x$. Putting everything together, we see that the number of $n \leq x$ with $f(n) \in \E$ is indeed $o(x)$ as $x\to\infty$.

One could also consider the function $f(n) = \ell(p_{n+1})$ whose values are the order of $2$ modulo the odd primes. Theorem \ref{thm:genCE} applies to this $f$ as well. The proof is similar to that just given, but slightly simpler, and we leave the details to the reader. In all of these statements, the obvious analogues hold with $2$ replaced by any fixed integer $a \not\in \{0, \pm 1\}$.

\subsection{Can we sum divisors properly?} It is natural to wonder if Theorem \ref{thm:genCE} applies to the sum-of-\emph{proper}-divisors function $s(n)$, defined by $s(1)=1$ and $s(n)=\sigma(n)-n$ for $n > 1$. The growth conditions are not difficult to check, but we do not know how to show that meager sets have density zero preimages. This seems to merit further study. We note that the stronger conjecture that any density zero set has a density zero preimage has been proposed by Erd\H{o}s, Granville, Pomerance, and Spiro \cite[Conjecture 4]{EGPS90}.

\section*{Acknowledgments}
The authors thank Greg Martin and Carl Pomerance for valuable suggestions.

\end{document}